\documentclass{amsart}

\usepackage{amssymb}
\usepackage{amsthm}
\usepackage{mathrsfs}
\usepackage{enumerate} 
\usepackage{color}
\usepackage{cancel}

\usepackage{rotating}

\usepackage[all]{xy}

\usepackage{fancyhdr}
\usepackage{xr-hyper}

\usepackage[colorlinks=true]{hyperref}

\usepackage{enumitem}
\newlist{BC}{enumerate}{1}
\setlist[BC]{label=(BC-\arabic*)}
\newlist{AC}{enumerate}{1}
\setlist[AC]{label=(AC)}

\newtheorem{theorem}{Theorem}[section] %was[subsection] before
\newtheorem{proposition}[theorem]{Proposition}
\newtheorem{lemma}[theorem]{Lemma}

\theoremstyle{definition}
\newtheorem{definition}[theorem]{Definition}

\newtheorem{notation}[theorem]{Notation}

\theoremstyle{remark}
\newtheorem{remark}[theorem]{Remark}

\numberwithin{equation}{section}

\newcounter{myenum}

\makeatletter
\renewcommand\thetheorem{\@arabic\c@section.\@arabic\c@theorem}
\newcounter{subtheorem}
 \@addtoreset{subtheorem}{theorem}
\renewcommand\thesubtheorem{\thetheorem.\@arabic\c@subtheorem}

\makeatother

\newcommand{\cal}{\mathcal}
\newcommand{\bb}{\mathbb}

\newcommand{\comment}[1]{}

\newcommand{\xym}{\xymatrix}

\newcommand{\into}{\hookrightarrow}
\newcommand{\PP}{\bb{P}}

\DeclareMathOperator{\pic}{Pic}
\DeclareMathOperator{\divfcn}{div}

\DeclareMathOperator{\charct}{char}

\DeclareMathOperator{\Bl}{Bl}

\newcommand{\oline}{\overline}
\newcommand{\ol}{\overline}
\newcommand{\wt}{\widetilde}

\newcommand{\bu}{\bullet}
\newcommand{\ul}{\underline}

\begin{document}

\title{Intersections via resolutions}
\author{Joseph Ross}
\address{University of Southern California Mathematics Department, 3620 South Vermont Avenue, Los Angeles CA, 90089}
\email{josephr@usc.edu}
\urladdr{\url{http://www-bcf.usc.edu/~josephr/}}

%
%\subjclass[2010]{Primary xxAxx}


\date{\today}

\maketitle
\begin{abstract}
We investigate the viability of defining an intersection product on algebraic cycles on a singular algebraic variety by pushing forward intersection products formed on a resolution of singularities.
For varieties with resolutions having a certain structure (including all varieties over a field of characteristic zero), we obtain a stratification which reflects the geometry of the centers and the exceptional divisors.
This stratification is sufficiently fine that divisors can be intersected with $r$-cycles (for $r \geq 1$), and 2-cycles can be intersected on a fourfold, provided their incidences with the strata are controlled.  Similar pairings are defined on a variety with one-dimensional singular locus.
\end{abstract}

\tableofcontents

\section{Introduction}

The (Borel-Moore) homology of a smooth manifold possesses a canonical ring structure in which the product is represented by intersection of cycle classes.
The Chow groups of a smooth algebraic variety over the complex numbers also admit a ring structure, and the intersection products are compatible via the cycle class map \cite[Cor.~19.2]{Ful}.  In both contexts, a well-behaved intersection theory for cycles fails to extend to spaces with singularities.

In topology, this motivated the definition of cohomology and its cup product operation, which in some sense isolates a subset of the space of cycles (with a different equivalence relation) which can be intersected even if the ambient space has singularities.  There are several analogues in algebraic geometry.
The Friedlander-Lawson theory of algebraic cocycles \cite{FLcocycle} is a geometric approach built from finite correspondences to projective spaces, and similar constructions underlie the motivic cohomology of Friedlander-Voevodsky \cite {FV}.  
The operational Chow cohomology of Fulton \cite{Ful}, on the other hand, adopts a more formal approach.

The intersection homology groups of Goresky and MacPherson \cite{GM1} provide an
interpolation between cohomology and homology: at least for a normal space $X$ there is a sequence of groups:
$$H^{\dim X -*}(X) = IH^{\ol 0}_* (X) \to \cdots \to IH^{\ol p}_*(X) \to \cdots \to IH^{\ol t}_*(X) = H_*(X)$$
factoring the cap product map.  The decoration $\ol p$, called the perversity, is a sequence of integers which prescribes how cycles may meet the strata in a suitable stratification of the (possibly singular) space $X$; in the display above it increases from left to right.  Each intersection homology group $IH^{\ol p}_r(X)$ arises as the homology of a complex of chains (either simplicial chains with respect to a triangulation \cite{GM1}, or singular chains \cite{King}) of perversity $\ol p$.

One of the most interesting features of this theory is the existence of intersection pairings
$$IH^{\ol p}_r(X) \otimes IH^{\ol q}_s(X) \to IH^{\ol p + \ol q}_{r+s - \dim(X)}(X)$$
(provided $\ol p + \ol q \leq \ol t$) generalizing the cap product pairing between cohomology and homology, and providing a generalization of Poincar\'e duality to singular spaces.

The author and Eric Friedlander \cite{intsing} have defined an algebraic cycle counterpart to the geometric approach of Goresky-MacPherson.  In particular we have defined perverse Borel-Moore motivic homology groups $H^{\ol p}_{m}(X, \bb{Z}(r))$ (for a stratified variety $X$ and a perversity $\ol p$) with a cycle class map to the Goresky-MacPherson theory in Chow degree: $H^{\ol p}_{2r}(X, \bb{Z}(r)) \to IH^{\ol p}_{2r}(X, \bb{Z})$ \cite[Defn.~2.3, Prop.~2.5]{intsing}.  We have constructed also a perverse variation of motivic cohomology $H^{i,s, \ol p}(X)$ and pairings \cite[Defn.~5.3, Prop.~6.13]{intsing}:
$$\cap : H^{i,s, \ol p}(X) \otimes H^{\ol q}_m (X, \bb{Z}(r) ) \to H^{\ol p + \ol q}_{m-i}(X, \bb{Z}(r-s)) .$$
There is a canonical morphism $H^{i, s, \ol p}(X) \to H^{\ol p}_{2 d -i}(X, \bb{Z}(d -s))$ \cite[Cor.~6.14]{intsing} with $d = \dim (X)$.  To define an intersection product
$$H^{\ol p}_{m}(X, \bb{Z}(r)) \otimes H^{\ol q}_{n}(X, \bb{Z}(s)) \to H^{\ol p + \ol q}_{m+n - 2 d}(X, \bb{Z}(r+s - d))$$
(i.e., extend $\cap$), one needs to involve the stratification in a more substantial way.

In this paper we approach the problem of defining such a product in Chow degree ($m=2r, n = 2s$) by considering the following (somewhat vague) question: if a singular variety $X$ can be resolved by a smooth variety $\wt X$, to what extent does the intersection product on the Chow groups of $\wt X$ provide a sensible intersection theory for algebraic cycles on $X$?  As one can already see from the case of a proper birational morphism $\pi : \wt X \to X$ between smooth varieties, the push-forward of the intersection formed on $\wt X$ is in general different from the intersection formed on $X$.
On the other hand, one also sees that for a blowup $\wt X \to X$ of a smooth variety $X$ along a smooth subvariety, the push-forward of the intersection formed on $\wt X$ agrees with the intersection formed on $X$ if the cycles have controlled incidence with $Y$; see Proposition \ref{smooth case} for a precise statement.

Our main results establish cases in which intersections formed on a resolution provide well-defined products on modifications of Chow groups: instead of considering cycles modulo rational equivalence, certain strata are singled out by the resolution, and both the cycles and the equivalences among them are required to have controlled incidence with the strata.  If $k$ is a field of characteristic zero, then any $k$-variety $X$ may be resolved by a sequence of blowups along smooth centers $\pi : \wt X \to X$.  From such a resolution we define a stratification of $X$ and construct an intersection pairing on perverse Chow groups (i.e., a map $A_{r, \ol p} (X) \otimes A_{s, \ol q}(X) \to A_{r+s-d}(X)$, possibly with $\bb{Q}$ coefficients) in the following settings:
\begin{itemize}
\item $s=\dim(X)-1$, i.e., one of the cycles is a divisor (Theorem \ref{divisor pairing});
\item $r=s=2$ and $\dim(X)=4$, i.e., there is an intersection pairing on 2-cycles on a fourfold (Theorem \ref{4fold}); and
\item $\dim(X_{sing})=1$ and $r,s$ arbitrary (Theorem \ref{one dim sing}).
\end{itemize}
The pairing is obtained by pushing forward the intersection product of the proper transforms, i.e., is given by $\alpha, \beta \mapsto \pi_* (\wt \alpha \cdot \wt \beta)$.  The stratification is expressed in terms of the geometry of the resolution, and our arguments are accordingly geometric in nature.  In all three cases, the basic idea is to understand how (certain) rational equivalences behave under proper transform.  For concreteness we explain the idea for 2-cycles $\alpha, \beta$ of complementary perversities $\ol{p}, \ol{q}$ on a fourfold $X$.  If one has $\alpha \sim_{\ol p} \alpha'$ on $X$, then $e_\alpha := \wt \alpha - \wt {\alpha'}$ is a cycle on $\wt X$ supported over the singular locus of $X$, with the support controlled by the perversity function $\ol p$.
The perversity $\ol{p}$ provides enough control over the error term $e_\alpha$ that we can find a cycle which is both rationally equivalent to some multiple of $e_\alpha$ and, using the complementarity of the perversity condition $\ol q$, disjoint from $\widetilde{\beta}$ (or $e_\beta$).
Then, at least with rational coefficients, we have the vanishings  $e_\alpha \cdot \wt \beta = \wt \alpha \cdot e_\beta = e_\alpha \cdot e_\beta = 0$ in $A_0(\wt X)$, so that $\pi_* (\wt \alpha \cdot \wt \beta )$ and $\pi_*(\wt \alpha' \cdot \wt \beta')$ coincide as classes in $A_0(X)_\bb{Q}$.
For divisors the arguments are more conceptual and less technical, but the main point is to move error terms arising from (certain) rational equivalences away from cycles of complementary perversity.  When the singular locus of $X$ is one-dimensional, we also employ intersection theory on the exceptional components and subvarieties therein.

For certain pairs of cycles, our product agrees with the one defined by Goresky-MacPherson via the cycle class map, but we do not know if this holds in general.  Another basic question is the dependence on the resolution.  See \S \ref{sec:further} for further discussion of these questions.

In \S \ref{sec:cycles} we establish background and notation.  In particular, we discuss various procedures for constructing stratifications from resolutions; later we point out which constructions are necessary in the different contexts.  The next three sections are devoted to the three situations mentioned above: in \S \ref{sec:divisors} we show the resolution provides a sensible intersection theory for divisors, in \S \ref{sec:onedimsing} we handle the case $\dim(X_{sing}) =1$, and in \S \ref{sec:4fold} we study 2-cycles on a fourfold.  In \S \ref{need integral fibers} we show, through an explicit example, that an ``obvious" coarsening of our stratification, namely by the fiber dimension of the resolution, is insufficiently fine for the resolution to provide a decent intersection product.

\medskip
\textbf{Conventions.} Throughout we work with schemes separated and of finite type over a field $k$.  A variety is an integral $k$-scheme.

\medskip
\textbf{Acknowledgments.} 
The idea that resolutions might provide pairings on perverse cycle groups emerged in conversations with Eric Friedlander.
The author wishes to thank him for his encouragement and interest in this project.  The author was partially supported by National Science Foundation Award DMS-0966589.

\section{Cycles, equivalence relations, and stratifications from resolutions} \label{sec:cycles}

In this section we adapt some of the basic features of Goresky-MacPherson intersection homology to our algebraic context.  This reproduces some material from \cite[\S 2]{intsing}.  Then we discuss stratifications obtained from resolutions.  Finally we demonstrate that our proposal for intersections on a singular variety recovers the usual intersection theory for the blowup of a smooth variety along a smooth subvariety.

A \textit{stratified variety} is a variety $X$ (say of dimension $d$) equipped with a filtration by closed subsets $X^d \into X^{d-1} \into \cdots \into X^2 \into X^1 \into X$ such that $X^i$ has codimension at least $i$ in $X$.  This is usually called a filtered space, which determines a stratified space; since we do not deal with more general stratifications, we ignore the difference.
A \textit{perversity} is a non-decreasing sequence of integers $p_1, p_2, \ldots, p_d$ such that $p_1=0$ and, for all $i$, $p_{i+1}$ equals either $p_i$ or $p_i+1$.  Perversities are denoted $\ol{p}, \ol{q}$, etc.  The perversities we consider range from the zero perversity $\ol{0}$ with $p_i=0$ for all $i$, to the top perversity $\ol{t}$, with $p_i=i-1$ for all $i$.

Let $Z_r(X)$ denote the group of $r$-dimensional algebraic cycles on $X$.  If $X$ is stratified and $\ol p$ is a perversity, we say an $r$-cycle $\alpha$ is \textit{of perversity $\ol{p}$} (or satisfies the perversity condition $\ol{p}$) if for all $i$, the dimension of the intersection $|\alpha| \cap X^i$ is no larger than $r-i +p_i$.  If the codimension of $X^i$ in $X$ is exactly $i$, then the perversity of a cycle measures its failure to meet properly the closed sets occurring in the stratification of $X$.  Let $Z_{r, \oline{p}}(X) \subset Z_r(X)$ denote the group of $r$-dimensional cycles of perversity $\ol{p}$ on the stratified variety $X$.  Typically $X^1$ is the singular locus of $X$, and then the condition $p_1=0$ means that no component of the cycle is contained in the singular locus.

In \cite[Prop.~2.4]{intsing} we defined and characterized a notion of $\ol p$-rational equivalence of algebraic $r$-cycles: we identify two elements of $Z_{r, \ol p}(X)$ if they can be connected by an $\bb{A}^1$-family of $r$-cycles of perversity $\ol p$.  Here we introduce a modification.
We say the cycles $\alpha, \alpha' \in Z_{r, \oline{p}}(X)$ are \textit{weakly rationally equivalent as perversity $\oline{p}$ cycles} (or are weakly $\ol{p}$-rationally equivalent), written $\alpha \sim_{w, \oline{p}} \alpha'$, if there is an equation of rational equivalence all of whose terms satisfy the perversity condition $\ol{p}$.  Explicitly this means there exist subvarieties $W_1, \ldots, W_a$ of dimension $r+1$ in $X$, and rational functions $g_i : W_i \dashrightarrow \PP^1$, such that $\alpha - \alpha' = \sum_i [g_i(0)] - [g_i(\infty)]$ in $Z_{r, \ol{p}}(X)$ (i.e., for all $i$, both $[g_i(0)]$ and $[g_i(\infty)]$ are in $Z_{r, \ol{p}}(X)$).  

In contrast to the definition of $A_{r, \ol p}(X)$, here we do \textit{not} require that for every $t \in \bb{A}^1 \subset \PP^1$, the cycle associated to the fiber $[g_i(t)]$ satisfies the perversity condition.  We allow the $g_i$'s to have ``bad" fibers, but we require that $\alpha$ and $\alpha'$ are related by ``good" fibers.  We could equivalently require that the condition of \cite[Prop.~2.4(3)]{intsing} is satisfied with $\bb{A}^1$ replaced by an open subset (containing at least two $k$-points) therein.
We let $A^w_{r, \ol{p}}(X)$ denote the group of $r$-cycles of perversity $\ol{p}$ modulo weak $\ol{p}$-rational equivalence.  
The main reason for introducing another equivalence relation is that our arguments ``naturally" construct pairings on the groups $A^w_{r, \ol p}(X)$.
Since there is a canonical morphism $A_{r, \ol{p}}(X) \to A^w_{r, \ol p}(X)$, we immediately obtain pairings on the groups $A_{r, \ol p}(X)$.  Note also the map $A_{r, \ol p}(X) \to IH^{\ol p}_{2r}(X, \bb{Z})$ factors through $A^w_{r, \ol p}(X)$: the proof of \cite[Prop.~2.5]{intsing} only requires $\cal W \into X \times \PP^1$ to have good fibers over $0, \infty \in \PP^1$, not all $t \in \PP^1$.

\medskip

\textbf{Resolutions.} Let $X$ be a variety over a field $k$ of characteristic zero.  The celebrated result of Hironaka asserts the singularities of $X$ can be resolved by a sequence of blowups along smooth centers which are normally flat in their ambient spaces, i.e., the resolution can be expressed as a composition $\pi: \widetilde{X} = X_f \to \cdots \to X_{n+1} \to X_n \to \cdots \to X = X_0$ where (for all $n = 0, \ldots, f-1$) $X_{n+1} = \Bl_{C_n} X_n$, $C_n \into X_n$ is smooth, and $X_n$ is normally flat along $C_n$ \cite{Hironaka}.  We will call a resolution admitting such an expression a \textit{strong resolution.}

\begin{definition} The exceptional locus $\widetilde{E} \into \widetilde{X}$ of the resolution is simply the preimage of the singular locus of $X$.  We say a subvariety $V \into \widetilde{X}$ is \textit{exceptional} if it is contained in the exceptional locus of the resolution.  We say an exceptional subvariety $V \into \widetilde{X}$ \textit{first appears} on $X_{n+1} = \Bl_{C_n}X_n$ if there is a subvariety $\ol{V} \into X_{n+1}$ such that the morphism $V \into \wt{X} \to X_{n+1}$ factors through a birational morphism $V \to \ol{V}$, and there is no such subvariety of $X_n$.  The concept of first appearance depends on the sequence of centers used to construct the resolution. \end{definition}

\begin{notation} If $C_n \into X_n$ is a center occurring in the resolution, let $\widetilde{E}_{C_n} \subseteq \widetilde{E}$ denote the union of those exceptional divisors $E$ for which the canonical morphism $E \to X$ factors through $C_n$, and let $\widetilde{E}_{\text{dom }C_n} \subseteq \widetilde{E}_{C_n}$ denote the union of those divisors for which the canonical morphism $E \to C_n$ is dominant.   \end{notation}

\begin{definition}[stratification via resolution] \label{strat res defn} 
We describe several ways a center $C_n$ may contribute to a stratification of the variety being resolved.
First we consider what happens ``below" a center $C_n$.  
\begin{BC}
\item If $C_n \to X$ has generic dimension $g$ over its $m$-dimensional image $W \into X$, then $W \into X^{d-m}$; and for each $e \geq 1$ the locus $W' \into W$ over which $C_n \to W$ has fiber dimension $\geq g+e$ is placed in $X^{d-m+e+1}$.  
\item If $g \geq 1$, the generic fiber of $C_n \to W$ is integral; the locus over which $C_n \to W$ fails to have an integral fiber of dimension $g$ is placed in $X^{d-m+1}$.
\item The singularities of $W$ are placed in $X^{d-m+1}$.
\end{BC}
Next we consider what happens ``above" a center.
\begin{AC}
\item For a general point $c \in C_n$, the number of irreducible components of ${(\widetilde{E}_{\text{dom }C_n})}_c$ is equal to the number of components of $\widetilde{E}_{\text{dom }C_n}$; where the former exceeds the latter is a closed set denoted $R_n \into C_n$.  
If the image $V$ of $R_n$ in $W$ is not dense, we place this image in $X^{d-m+e}$, where $e$ is the codimension of $V$ in $W$ (hence $d-m+e$ is the codimension of $V$ in $X$). 
\end{AC}
Let $\pi : \wt X \to X$ be a resolution expressed as a sequence of blowups along $C_n \into X_n$.  For any of the four constructions listed above, we say a stratification of $X$ \textit{satisfies condition (C) with respect to $\pi$} if it refines the stratification obtained by applying construction (C) to all of the centers $C_n$ occurring in the resolution.  \end{definition}

\begin{remark} Let $\pi : \wt X \to X$ be a resolution.  If a stratification satisfies (BC-1) with respect to $\pi$, then it refines the ``fiber dimension" stratification given by specifying $X^i$ is the closed subset along which the fibers of $\pi$ have dimension at least $i-1$.

A stratification satisfying (BC-1), (BC-2), (BC-3), and (AC) may not have the property that $X^i \setminus X^{i+1}$ is smooth since we do not, for example, necessarily place the incidences of the components of $X_{sing}$ in a smaller stratum.  For an explicit example in which the incidences of the components of the singular locus are not detected by the conditions above, see the projective example of Section \ref{need integral fibers}.

It would be very interesting to characterize intrinsically the strata which arise from the conditions described above.
\end{remark}

\medskip
\textbf{Error terms.} Suppose $\alpha \sim_{w, \ol{p}} \alpha'$ and \begin{equation} \label{alpha equiv} \alpha - \alpha' = \sum_{f_i \in k(S_i)} [f_i(0)] - [f_i(\infty)] \end{equation}
is an equation of rational equivalence in $Z_{r, \ol{p}}(X)$.  Viewing the $f_i$'s as rational functions on the proper transforms of the $S_i$'s, we obtain an equation of rational equivalence $\widetilde{\alpha} + e_\alpha - \widetilde{\alpha'} - e_{\alpha'} = \sum_{f_i \in k(\widetilde{S_i})} [f_i(0)] - [f_i(\infty)]$ in $Z_r(\widetilde{X})$, where $e_\alpha$ is exceptional for some $\widetilde{S_i} \to S_i$ and is supported over the image of the exceptional locus of $\widetilde{z} \to z$ for $z$ some term in the equation $\ref{alpha equiv}$ (and similarly for $e_{\alpha'}$).  In this situation we refer to $e_\alpha$ and $e_{\alpha'}$ as the ``error terms."
We write $e(\alpha, \alpha')$ for the union of the supports of the error terms, and typically $e_\alpha$ will denote a component of the support of $e(\alpha, \alpha')$.

\medskip
\textbf{Normality.} There is no loss of generality in assuming $X$ is normal.  For if $X$ is not necessarily normal and $\nu : X^\nu \to X$ is the normalization, and $\pi : \widetilde{X^\nu} \to X^\nu$ is the result of applying a resolution algorithm to $X^\nu$, we simply declare the stratification of $X$ to be the image via $\nu$ of the stratification of $X^\nu$ (induced by $\pi$), augmented by declaring $X^1$ is the singular locus of $X$.  Note that $\nu \circ \pi$ is probably not the result of applying a resolution algorithm to $X$ itself.

For $\alpha \in Z_{r, \ol{p}}(X)$, let $\alpha^\nu \in Z_r(X^\nu)$ denote its proper transform.  By definition, $\alpha \in Z_{r, \ol{p}}(X)$ implies $\alpha^\nu \in Z_{r, \ol{p}}(X^\nu)$.  If $\alpha \in Z_{r, \ol{p}}(X)$, then $\dim ( \nu^{-1}(\alpha \cap X^1)) = \dim (\alpha \cap X^1) = r-1$ since $\nu$ is finite and $\alpha$ is not contained in $X^1$ (since $p_1=0$).  Therefore $\ol{p}$-rational equivalence on $X$ implies $\ol{p}$-rational equivalence on $X^\nu$.  Let $\nu^* : A_{r, \ol{p}}(X) \to A_{r, \ol{p}}(X^\nu)$ denote the morphism induced by proper transform.

Now suppose we can define a pairing $A_{r, \ol{p}}(X^\nu) \times A_{s, \ol{q}}(X^\nu) \to A_{r+s-d}(X^\nu)$ when $\ol{p} + \ol{q} = \ol{t}$.  Then the composition
$$A_{r, \ol{p}}(X) \times A_{s, \ol{q}}(X) \xrightarrow{\nu^* \times \nu^*} A_{r, \ol{p}}(X^\nu) \times A_{s, \ol{q}}(X^\nu) \to A_{r+s-d}(X^\nu) \xrightarrow{\nu_*} A_{r+s-d}(X)$$
defines the pairing for $X$.

Moreover, the condition $p_1=0$ implies the generic point of any error term factors through $X^2$: we have $\dim (\alpha \cap X^1) \leq r-1$, and $\pi$ is finite over $X \setminus X^2$, hence $\dim (\pi^{-1} (\alpha \cap X^1 \setminus X^2)) \leq r-1$ and $\pi^{-1}(\alpha \cap X^1 \setminus X^2)$ cannot support any error terms.

If $X$ is not normal, the stratification ``prescribed by the resolution" is the one prescribed by $\widetilde{X^\nu} \to X^\nu \to X$.  We assume $X$ is normal and hence $X^2$ consists of the singular locus of $X$.

\medskip
The following proposition demonstrates our proposal is sensible when the ``resolution" is the blowup of a smooth variety along a smooth subvariety.  The method of proof also shows the condition $\ol p + \ol q \leq \ol t$ cannot in general be weakened.

\begin{proposition} \label{smooth case}
Let $Y \into X$ be a closed immersion of smooth varieties of codimension $c$.  Let $\alpha, \beta \into X$ be cycles of dimensions $r,s$ satisfying
$$\dim(\alpha \cap Y) \leq r-c+p \ , \ \dim(\beta \cap Y) \leq s-c+q,$$
with $p+q \leq c-1$.  Let $\pi: \wt X \to X$ be the blowup of $X$ along $Y$.  Then $\alpha \cdot \beta = \pi_*( \wt \alpha \cdot \wt \beta)$ as cycle classes on
$|\alpha| \cap |\beta|$ and hence on $X$.
\end{proposition}

\begin{proof} We have a cartesian diagram:

$$\xym{ E \ar[r]^-j \ar[d]_-g & \wt X\ar[d]^-\pi \\ Y \ar[r] & X .\\}$$
Let $z \in A_*(E)$ denote the class of the canonical $g$-ample line bundle.  Then $A_*(E)$ is $A_*(Y) [z]$ modulo a relation involving the Chern classes of the normal bundle $N_YX$; we will need that $g^*$ is a ring homomorphism, and that $g_* (z^n) = 0$ for $n \leq c-2$.

 We set $\pi^* \alpha = \wt \alpha + e_\alpha$ and $\pi^* \beta = \wt \beta + e_\beta$ in $A_*(\wt X)$; precise formulas for $e_\alpha, e_\beta$ are available \cite[Thm.~6.7]{Ful}, but we will just need that $\pi_* e_\alpha = \pi_*e_\beta=0$.  Since $\alpha \cdot \beta = (\pi_* \pi^* \alpha) \cdot \beta = \pi_* (\pi^* \alpha \cdot \pi^* \beta)$, it suffices to show the terms $\wt \alpha \cdot e_\beta, e_\alpha \cdot \wt \beta,$ and $e_\alpha \cdot e_\beta$ vanish after $\pi_*$ is applied.

Now we analyze the term $\wt \alpha \cdot e_\beta$; the treatment of the term $e_\alpha \cdot \wt \beta$ merely involves interchanging the roles of $\alpha$ and $\beta$.  We utilize the Chow rings of $Y$ and $E$ and show the vanishing of $g_* ( j^* (\wt \alpha) \cdot e_\beta )$, which implies $\pi_* (\wt \alpha \cdot e_\beta) = 0$.  Since $j^* (\wt \alpha)$ is an $(r-1)$-dimensional cycle supported over the ($\leq (r-c+p)$-dimensional) set $\alpha \cap Y$, it can be expressed as
$$j^* (\wt \alpha) = z^p g^*( a_{r-c+p} ) + z^{p-1}  g^* (a_{r-c+p-1} ) + \cdots +  g^*( a_{r-c} )$$
where $a_i$ is a class in $A_i(Y)$.  Similarly, since $e_\beta$ is an $s$-dimensional cycle supported over $\beta \cap Y$, we have an expression
$$e_\beta = z^{q-1} g^* (b_{s-c+q}) + z^{q-2} g^* (b_{s-c+q-1}) + \cdots + g^* ( b_{s-c+1})$$
with $b_i \in A_i(Y)$.  Therefore $ j^* (\wt \alpha) \cdot e_\beta$ is a sum of terms of the shape $z^n g^* ( c_n )$ (with $c_n \in A_{r+s-d+1+n-c}(Y)$) with $n \leq p+q -1$.  Now the projection formula for $g$, the formula $g_* (z^n) = 0$ for $n \leq c-2$, and the hypothesis $p+q \leq c-1$ together imply that $g_* ( j^* (\wt \alpha) \cdot e_\beta) =0 $, as desired.

To show the vanishing of the term $e_\alpha \cdot e_\beta$, we note that
$$e_\alpha \cdot_{\wt X} e_\beta =  j_* ( c_1(N_{E} \wt X) \cap (e_\alpha \cdot_E e_\beta) ),$$
where we have used a subscript to indicate on which variety we calculate the product.  The largest $z$-degree in $e_\alpha \cdot_E  e_\beta$ is $p+q-2$, and $ c_1(N_{E} \wt X) = -z$, so the same argument as in the previous paragraph shows the vanishing of $g_* ( c_1(N_{E} \wt X) \cap (e_\alpha \cdot e_\beta) )$, hence $\pi_* (e_\alpha \cdot_{\wt X} e_\beta) = 0$.
\end{proof}

\begin{remark}
There does not seem to be a simple inductive argument which allows one to conclude Proposition \ref{smooth case} holds for a composition of blowups along smooth centers $\pi: X_2 = \text{Bl}_{Y_1}(X_1) \xrightarrow{\pi_2} X_1 = \text{Bl}_{Y_0}(X_0) \xrightarrow{\pi_1} X_0 =X$ for the following reason: if cycles on $X_0$ have controlled incidence with $Y_0$, it is not necessarily the case that their proper transforms via $\pi_1$ will have suitably controlled incidence with $Y_1 \into X_1$.
For example, suppose $Y_0$ is a linearly embedded $\PP^2$ in $X_0 = \PP^4$, and $Y_1$ is the preimage ${(\pi_1)}^{-1}(L)$ of a line $L \into Y_0$ via the blowup.
Then $Y_1 \to L$ is a smooth morphism between smooth varieties, and has geometrically integral fibers.  The only sensible stratification is
$$\emptyset = X^4 \into L = X^3 \into Y_0 = X^2 = X^1 \into X ;$$
since no lines are distinguished and no line has a distinguished point, there is no canonical way to define a non-empty finite set $X^4$.  Now if $S \into X$ satisfies $\dim (S \cap X^2) = \dim (S \cap X^3) = 0$, then $S$ satisfies the perversity condition $ \ol p = (0, 1, 1)$ (and hence also the condition $\ol q = (1,1,2)$).  Therefore the pair $(S,S)$ satisfies a pair of complementary perversity conditions.   However, the proper transform $S_1 \into X_1$ satisfies $\dim (S_1 \cap Y_1) = 1$, so that $(S_1, S_1)$ violates the condition $p + q \leq 1$ for the morphism $\pi_2$ required by Proposition \ref{smooth case}.

The self-intersection $S \cdot S$ agrees with $\pi_* (S_2 \cdot S_2)$,
but the ``reason" involves a feature of the morphism $X_2 \to X_0$ which cannot be seen from its constituent factors: the ``error term" $\pi^* (S) - S_2$ is represented by surfaces supported over the finite set $S \cap X^3$, and these can be moved away from each other (and from $S_2 \cap \pi^{-1}(S \cap X^3)$) via rational functions on $L$.
In fact the same procedure works in the singular case, with possible modifications due to the more complicated geometry; see the penultimate paragraph of the proof of Proposition \ref{intro smooth ex} for details.
\end{remark}

\section{Intersection with divisors} \label{sec:divisors} In this section we construct a pairing between divisors and $r$-cycles (for $r \geq 1$) by using a stratification obtained from a resolution of singularities.  We make use of the ring structure on $A_* (X)$ for nonsingular $X$, and we use that products respect supports in the following sense: if $A, B \into  X$ are cycles of dimensions $r,s$, then $A \cdot B \in A_{r+s-d}( |A | \cap | B|)$ is a well-defined cycle class \cite[Ch.~8]{Ful}.  Often we use that $|A| \cap |B|$ (or its image via a resolution) cannot support a cycle of the relevant dimension.

We use $(-)^i$ to indicate the intersection $|(-)| \cap |X^i|$ in the remainder of the paper.

\begin{theorem} \label{divisor pairing}
Let $X$ be a $d$-dimensional variety over $k$, and let $\pi : \widetilde{X} \to X$ be a strong resolution of singularities.  
Suppose a stratification of $X$ satisfies (BC-1) and (AC) with respect to $\pi$.
Let $\oline{p}, \oline{q}$ be perversities such that $\oline{p} + \oline{q} = \oline{t}$.
Then the assignment
$$(D, \alpha) \mapsto \pi_* (\widetilde{D} \cdot \widetilde{\alpha})$$
determines a well-defined pairing
$$A^w_{d-1, \oline{p}}(X)_\bb{Q} \otimes A^w_{r, \oline{q}}(X)_\bb{Q} \to A_{r-1}(X)_\bb{Q}$$
for any $r \geq 1$.
\end{theorem}

\begin{proof} First we show the assignment is compatible with $\ol{p}$-rational equivalences of divisors.  If $D$ and $D'$ are divisors with $D \sim_{w, \oline{p}} D'$, then in particular $D - D' = \divfcn (f)$ for some $f \in k(X)$.  Viewing instead $f \in k (\widetilde{X})$, we find $\widetilde{D} - \widetilde{D'} = e_D + \divfcn(f)$, where the error term $e_D$ is an exceptional divisor for $\pi$, and is supported over $(D \cup D') \cap X^2$.  Since $\widetilde{\alpha}$ is not contained in the exceptional locus, the intersection $\widetilde{\alpha} \cap e_D$ is proper.

Since a stratification satisfying (BC-1) refines the fiber dimension stratification, the preimage $\pi^{-1}(T)$ of a divisor $T \into X^i \setminus X^{i+1}$ has dimension no larger than $(d-i-1)+(i-1) = d-2$.  Therefore $e_D$ is supported over the strata for which $p_i \geq 1$.   For these strata, $q_i \leq i-2$, so $\dim (\alpha^i) \leq r -2$.  But now $\pi_* ( e_D \cdot \widetilde{\alpha})$ is an $(r-1)$-cycle supported in a subscheme of dimension $r-2$, hence $\pi_* ( e_D \cdot \widetilde{\alpha}) = 0$.

The vanishing of $\pi_* (e_D \cdot e_\alpha)$ is proved by a similar argument: if $e_\alpha \not \subseteq e_D$, then an identical argument applies.  If $e_\alpha \subseteq e_D$, then $e_\alpha$ is supported over a subscheme of dimension at most $r-2$ (since it is supported over those strata for which $q_i \leq i-2$), so that $e_D \cdot e_\alpha$ is represented by a cycle having positive generic dimension over its image.

Now we show the assignment is compatible with $\ol{q}$-rational equivalences of $r$-cycles, i.e., the vanishing $\pi_* (\wt D \cdot e_\alpha)= 0$.  If $\alpha \sim_{w, \oline{q}} \alpha'$, then $\widetilde{\alpha} - \widetilde{\alpha'} \in A_r(\widetilde{X})$ is represented by a cycle which is supported over $A \cap X^2$, where $A$ is an $r$-cycle of perversity $\oline{q}$.  (Namely $A$ consists of the union of the supports of the terms appearing in the equation of rational equivalence, i.e., the image via $\pi$ of $e(\alpha, \alpha')$.)

First assume $q_2 = 0$, so $\alpha$ and $\alpha'$ meet the singular locus $X^2$ properly.  Then any component $e_\alpha$ of $e(\alpha, \alpha)$ is an $r$-cycle having generic dimension at least $2$ over its image, and this image is contained in the (at most) $(r-2)$-dimensional set $A \cap X^2$.  Therefore $ \wt D \cdot e_\alpha$ is represented by an $(r-1)$-cycle having generic dimension at least $1$ over its image, and so $\pi_* ( \wt D \cdot e_\alpha) =0$.

Finally assume $p_2 = 0$; this is the most interesting case.  We may assume $e := e_\alpha$ is integral and first appears on $X_{n+1} = \Bl_{C_n} X_n$.  Let $\ol{e} \into X_{n+1}$ denote the ($r$-dimensional) image of $e$, and let $e_n \into C_n \into X_n$ denote the image of $e$.  Since $e$ first appears on $X_{n+1}$, there is an exceptional divisor $E'$ for the morphism $X_{n+1} \to X_n$ which contains $\ol e$ and dominates $C_n$.  Since $e \to e_n$ is birational, there exists a component $E \subseteq \wt E \into \wt X$ which contains $e$ and maps birationally onto $E'$.  (Since  the exceptional locus of $X_{n+1} \to X_n$ is flat over $C_n$, every exceptional component of $X_{n+1} \to X_n$ must dominate $C_n$.)
Let $R_n \into C_n$ denote the locus determined by the reducible fibers over $C_n$ as described in condition (AC) in Definition $\ref{strat res defn}$; we will use that $R_n$ contains the reducible fiber locus of $E \to C_n$.  Let $V \into W$ denote the image of $R_n$ in $X$.  

Note that $\dim (e_n) \leq r-1$, hence if $e_n$ has positive generic dimension over its image in $X^2$, then we may proceed as in the case $q_2=0$.  So we may assume $e_n$ is generically finite over its image $\ol{e_n} \into W \into X$.  Here $W$ denotes the image of $C_n$; say $\dim W = d-c$, so that $W$ is a component of $X^c \into X^2$.  We have $\dim ( \ol{e_n}) = r-1 = (r-c) + (c-1)$, hence $q_c = c-1$ and $p_c = 0$, so $D$ meets $W$ properly.

Let $D_n \into X_n$ denote the proper transform of $D$ via $X_n \to X$, and $\wt{D} \into \wt{X}$ its proper transform via $\wt{X} \to X$.  Since $D$ does not contain $W$, $D_n$ does not contain $C_n$ and $\wt{D}$ does not contain $E$.
The situation is summarized by the following commutative diagram:

$$\xym{ & e \ar[r] \ar[d] & E \ar[r]^-i \ar[d]_{\pi _E } & \widetilde{X} \ar[d]^-\pi  & \wt{D} \ar[l] \ar[d] \\
R_n \ar@/_1pc/[rr]|\hole  \ar[d] & e_n \ar[r] \ar[d] & C_n \ar[d] \ar[r]^-{i_n}  & X_n \ar[d] & D_n \ar[l] \ar[d] \\
V \ar@/_1pc/[rr] & \ol{e_n} \ar[r] &  W \ar[r] & X & D \ar[l]  }$$

\medskip

We assume that $R_n$ is not contained in $D_n$, and later we handle the case $D_n \supseteq R_n$ by a separate argument.  
A priori we have  $\pi_E^{-1} |D_n \cap C_n| \supseteq | \widetilde{D}  \cap E |$, but $D_n \not \supset R_n$, together with the condition on fiber integrality, implies the generic points of $\pi_E^{-1} |D_n \cap C_n|$ coincide with those of $|\widetilde{D} \cap E|$. 
Since $C_n$ is smooth, $D_n \cap C_n$ is the support of a Cartier divisor in $C_n$.  Set $M := \cal{O}_{C_n} (|D_n \cap C_n|) \in \pic(C_n)$, and $L := \cal{O}_{\widetilde{X}} (\widetilde{D})  \in \pic(\widetilde{X})$.  The coincidence of the generic points implies $i^*L$ and ${\pi_E}^* M$ are rational multiples of one another in $\pic (E)_\bb{Q}$; say $m \cdot {\pi_E}^* M = i^* L$ for some $m \in \bb{Q}$.

Since ${\pi_E}_* (e) = 0$, the projection formula for $\pi_E$ implies
\begin{equation} \label{push 0} {\pi_E}_* (c_1 ( {\pi_E}^*(M)) \cdot e) = c_1(M) \cdot {\pi_E}_*(e) = 0. \end{equation}
The projection formula for $i$ implies the equality of cycle classes
\begin{equation} \label{proj i} i_* ( c_1 ( i^* L ) \cdot e) = c_1(L ) \cdot e \in A_{r-1}(\widetilde{X}). \end{equation}
The functoriality of proper push-forward and the relation in $\pic (E)_\bb{Q}$ imply
\begin{equation} \label{last one} \pi_* i_* ( c_1 ( i^* L ) \cdot e)  = {i_n}_* {\pi_E}_*  ( c_1 ( i^* L ) \cdot e) =  m \cdot {i_n}_* {\pi_E}_*  ( c_1({ \pi_E}^* ( M )) \cdot e). \end{equation}
These equations together imply the vanishing $\pi_* ( c_1(L) \cdot e) = 0$ in $A_{r-1}(X)_\bb{Q}$.

We return to the case in which $D_n \supseteq R_n$.  Then $D$ contains the image $V$ of $R_n$ in $X$, and $V$ is a component of $X^{d-c+m}$ (for some $m \geq 1$, since $p_c=0$).  Then $p_{d-c+m} \geq 1$, hence $q_{d-c+m} \leq d-c+m -2$ and therefore $\dim (A \cap V) \leq r-2$.
We have a commutative diagram relating push-forward and pull-back
$$\xym{A_{r-1} (\wt{X}) \ar[d] \ar[r]^-{\pi_*} & A_{r-1}(X) \ar[d]^-\cong \\
A_{r-1} (\wt{X} \setminus \pi^{-1} | A \cap V| ) \ar[r] & A_{r-1} (X \setminus | A \cap V|) }$$
so it suffices to show $\pi_* ( \wt{D} \cdot e) $ vanishes upon restriction to $X \setminus | A \cap V|$.  
Away from $V$, the generic points of $\pi_E^{-1} |D_n \cap C_n|$ coincide with those of $|\widetilde{D} \cap E|$ (as in the case $R_n \not \subset D_n$), and the vanishing of $\pi_* (\wt D \cdot e)$ in $A_{r-1}(X)$ follows.
\end{proof}
\begin{remark} Outside of the case $p_2=0$, one can work with the coarser fiber dimension stratification, and with integral coefficients.  \end{remark}

\section{Intersections on a variety with one-dimensional singular locus} \label{sec:onedimsing}

In this section we show a resolution may be used to defined intersection pairings on a variety with one-dimensional singular locus.  At one place we need the generic smoothness of a morphism between smooth integral $k$-schemes, so our results in this section are aimed at the case $\charct k = 0$, though they apply to special situations in positive characteristic.  We do not require a strong resolution, but we require the smoothness of the components of the exceptional locus $\wt E \into \wt X$.  Resolutions as in Theorem $\ref{one dim sing}$ exist by \cite[Thm.~1.6(2)]{BierMil}, or, since we do not use the smoothness of the centers, \cite{Kollar:res}.

\medskip
\textbf{The stratification.} Let $X$ be a variety over $k$.  Suppose $\pi : \wt X \to X$ is a resolution of the singularities of $X$ such that all exceptional components are smooth over $k$, and generically smooth over their images in $X$.  We set $X^{d-1}$ equal to the singular locus of $X$, and we define $X^d \into X^{d-1}$ to be the smallest set with the following properties:
\begin{enumerate}
\item $X^d$ contains the singularities of $X^{d-1}$, i.e., contains the singularities of each component of $X^{d-1}$, and contains the points at which the components intersect;
\item $X^d$ contains the image of every exceptional divisor $E \subseteq \wt E \into \wt X$ that is contracted to a point by the resolution $\pi : \wt X \to X$; and
\item for every exceptional divisor $E$ with 1-dimensional image $E_1$, suppose the morphism $E \to E_1$ has smooth generic fiber (e.g., $\charct k =0$); then $X^d$ contains the image of the singular fibers of $E \to E_1$.
\end{enumerate}

\begin{remark} 
Condition (1) is a mild strengthening of condition (BC-3) of Definition \ref{strat res defn}, since we additionally take into account the incidences of the components of $X^{d-1}$.  Condition (BC-1) implies condition (2) must be satisfied.  Condition (3) above is a strengthening of condition (AC).
\end{remark}

\begin{theorem} \label{one dim sing} Suppose $X$ is a variety over a field $k$ such that $\dim(X_{sing}) =1$, and suppose $\pi : \wt X \to X$ is a resolution of singularities such that the exceptional components are 
\begin{itemize}
\item$k$-smooth, and 
\item generically smooth over their images in $X$.
\end{itemize}
Let $\ol p, \ol q$ be perversities such that $\ol p + \ol q = \ol t$.
With respect to the stratification defined above, the assignment $\alpha, \beta \mapsto \pi_* ( \wt \alpha \cdot \wt \beta)$ determines a well-defined pairing $A^w_{r, \ol p}(X) \otimes A^w_{s, \ol q}(X) \to A_{r+s-d} (X)$.  \end{theorem}

As a matter of notation, we mostly let $\alpha$ denote the factor in which error terms are considered, so that $e_\alpha \in Z_*(\wt X)$ is a cycle which arises by taking the proper transform of an equivalence (respecting some perversity condition) relating (say) $\alpha$ to $\alpha'$ on $X$.

\begin{proof} The claim is obvious if $r+s-d \geq 2$, for then $e_\alpha \cdot \wt \beta, \wt \alpha \cdot e_\beta,$ and $e_\alpha \cdot e_\beta$ are all represented by cycles of dimension $\geq 2$ supported over $X^{d-1}$, hence have generic dimension $\geq 1$ over their images, so that all vanish after $\pi_*$ is applied.  This case requires no perversity condition at all.

Consider the case $r+s-d=1$.  If the error term $e_\alpha$ is supported over a finite set, then $e_\alpha \cdot \wt \beta$ is represented by a 1-cycle which is contracted to a finite set by $\pi$ (since $e_\alpha$ is so contracted), hence $\pi_* (e_\alpha \cdot \wt \beta) = 0$.  The same argument works with $e_\beta$ in the place of $\wt \beta$.  Therefore we may assume $\dim (\alpha^{d-1}) =1$, so that $p_{d-1} = d - r$ and $q_{d-1} \leq r-2$, and hence $\dim (\beta^{d-1})  \leq 0$.  Working one component at a time, we may assume $e_\alpha$ is contained in a single (smooth) exceptional component $i :E \into \wt X$.  Now $e_\alpha \cdot \wt \beta = i_* (e_\alpha \cdot i^* (\wt \beta))$, and $i^* (\wt \beta)$ is represented by an $(s-1)$-dimensional cycle supported over the finite set $\beta^{d-1}$.  Therefore $e_\alpha \cdot i^* (\wt \beta)$ is represented by a 1-cycle which is contracted to a finite set, so $ ( \pi \circ i)_* ( e_\alpha \cdot i^* (\wt \beta) ) = 0$.  Therefore $\pi_* (e_\alpha \cdot \wt \beta) = 0$, as desired.  Since one of $e_\alpha, e_\beta$ must be supported over a finite set in $X^{d-1}$, the term $\pi_* (e_\alpha \cdot e_\beta)$ vanishes for the same reason.

Now we consider the case $r+s-d=0$.  Now $\dim (\alpha^{d-1}) =1$ implies $\dim (\beta^{d-1}) < 0$, so the interesting situation is when both $\alpha^{d-1}$ and $\beta^{d-1}$ are finite; only one of $\alpha, \beta$ is allowed to meet $X^d$.  As in the previous paragraph, we choose a (smooth) exceptional component $i : E \into \wt X$ containing $e_\alpha$.  Both $e_\alpha$ and $i^* (\wt \beta)$ are supported over finite sets in $X^{d-1}$, and their incidence must occur over points in $X^{d-1} \setminus X^d$.  
Again working one component at a time, the cycles $e_\alpha$ and $i^* (\wt \beta)$ are disjoint unless they are supported over the same point $x \in X^{d-1} \setminus X^d$.  Let $j : E_x \into E$ denote the inclusion of the exceptional fiber over $x$.  By construction of the stratification, $E_x$ is smooth.  

Set $\beta' := i^* (\wt \beta)$, and let $N$ denote the normal bundle of the embedding $j$.  Using the projection formula and the self-intersection formula (\cite[Cor.~6.3]{Ful}), we find:
$${j}_* (e_\alpha) \cdot_E  {j}_* ( \beta' ) = j_* (e _\alpha \cdot_{E_x} j^* j_* \beta' ) = j_* (e_\alpha \cdot_{E_x} ( c_1 (N) \cap \beta')).$$
But $E_x$ is principal, hence $c_1 (N) = 0$, and therefore
$${j}_* (e_\alpha) \cdot_E  {j}_* (\beta') = e_\alpha \cdot_ E i^* (\wt \beta) =0.$$
The projection formula for $i$ implies $e_\alpha \cdot \wt \beta = 0 \in A_0( \wt X)$.

The vanishing of $e_\alpha \cdot e_\beta$ holds for a similar reason: the nontrivial case is when both $e_\alpha$ and $e_\beta$ are supported over the same point $x \in X^{d-1} \setminus X^d$.  But then
$j_* (e _\alpha) \cdot_E j_* (e_\beta) = j_* ( e_\alpha \cdot_{E_x}  ( c_1(N) \cap e_\beta ) )$, and the vanishing of $c_1(N)$ allows us to conclude.   \end{proof}

\begin{remark}
The basic obstacle to extending the above analysis to the case $\dim(X_{sing})=2$ is that the singular fibers of $E \to E_2$ (here $E$ is a component of the singular locus with 2-dimensional image $E_2 \into X$) may be supported over a divisor in $E_2$, and, in the case $\alpha^{d-1}$ and $\beta^{d-1}$ are both finite, the perversity conditions do not rule out the incidence being supported in a singular fiber.
\end{remark}

\section{2-cycles on a fourfold} \label{sec:4fold} In this section we work on a normal quasi-projective $4$-dimensional variety $X$; more precisely we assume the singularities of $X$ occur in codimension at least $2$.  As in the previous cases, the resolutions we require exist in characteristic zero.
We describe more explicitly the construction of the stratification satisfying (BC-1), (BC-2), (BC-3), and (AC) with respect to a strong resolution.  We start by declaring $X^2$ is the singular locus of $X$.

\begin{enumerate}

\item (BC-1) For every (smooth, integral) two-dimensional center $S$ such that the composition $S \to X^2$ is dominant over a component of $X^2$, the image of the positive-dimensional (i.e., one-dimensional) fibers of $S \to X^2$ is a finite set in $X^2$.  Place this set in $X^4$.

\item For every two-dimensional center $S$ such that the composition $S \to X^2$ has one-dimensional image $W \into X^2$,

\subitem (BC-1) place $W$ in $X^3$, and

\subitem (BC-3) place the singularities of $W$ in $X^4$.

\subitem (BC-2)  The morphism $p : S \to W$ has integral generic fiber.  Place the zero-dimensional set $ \{ w \in W | p^{-1}(w) \text{ is reducible} \}$ in $X^4$.
 
\item (BC-1) If the two-dimensional center $S$ of a blowup has zero-dimensional image in $X^2$ (so it lies over a single point $x \in X^2$), then $x$ is placed in $X^4$.

\item (AC) Let $p : \widetilde{E}_{\text{dom }S} \to S$ denote the canonical morphism.  Suppose $\widetilde{E}_{\text{dom }S}$ has $t$ components.  Then the image in $X^2$ of the closed set $\{ s \in S | p^{-1}(s) \text{ has more than $t$ components} \}$
is placed in $X^3$.
\item For every one-dimensional center $C$ which is generically finite onto its image in $X^2$,

\subitem (BC-1) place the image curve in $X^3$, and

\subitem (BC-3) place its singularities in $X^4$.

\subitem (BC-1) If a one-dimensional center $C$ has zero-dimensional image in $X^2$ (so it lies over a single point $x \in X^2$), then $x$ is placed in $X^4$.

\subitem (AC) Place in $X^4$ the closed set in $C$ over which $\widetilde{E}_{\text{dom }C} \to C$ has more components than does $\widetilde{E}_{\text{dom }C}$ itself.

\item (BC-1) The image in $X^2$ of every zero-dimensional center $Z$ is placed in $X^4$.

\end{enumerate}

The instances of condition (BC-1) are necessary to guarantee the stratification refines the stratification by fiber dimension.  In the course of the proof we point out where the other conditions are used; the condition (BC-3) in (5) does not seem to be necessary, but (BC-3) in (2) is used.

We will use notation from the following diagram.

$$\xymatrix{ \widetilde{X} \ar[r] & \Bl_{Z_n}X_n \ar[r]  & X_n \ar[r] & \ldots \ar[r] & X \\
\widetilde{E} \ar[u] & E_n \ar[u] \ar[r]^-{p_n} & Z_n \ar[u] & \ldots & X^2 \ar[u] \\ }$$
The center $Z_n$ will be written as $S_n$ when it is two-dimensional and as $C_n$ when it is one-dimensional.  Note that each $Z_n$ is smooth, and each $p_n : E_n \to Z_n$ is a flat morphism.  Zero-dimensional centers play no essential role since incidences in $X^4$ are forbidden by the perversity condition.
Blowups along three-dimensional centers are finite morphisms (by normal flatness) and do not influence the stratification.

\begin{theorem} \label{4fold} Let $X$ be a quasi-projective fourfold over $k$, and let $\pi : \wt X \to X$ be a strong resolution of singularities.
Suppose a stratification of $X$ satisfies (BC-1), (BC-2), (BC-3), and (AC) with respect to $\pi$.
Let $\ol{p}, \ol{q}$ be perversities such that $\ol{p} + \ol{q} = \ol{t}$.  The assignment $(\alpha, \beta) \mapsto \pi_* (\widetilde{\alpha} \cdot \widetilde{\beta})$ determines a well-defined pairing $A^w_{2, \ol{p}} (X)_\bb{Q} \otimes A^w_{2, \ol{q}} (X)_\bb{Q} \to A_0(X)_\bb{Q}.$
\end{theorem}

There are essentially three complementary pairs of perversities to analyze, and these are handled in the next three propositions.  For each pair the strategy is to consider possible locations of the generic points of error terms, and for each location we move the error term away from the other 2-cycle.  If the error term first appears on the blowup of $X_n$ along $Z_n$, then the move is achieved by finding a suitable rational function on $Z_n$.  

\begin{proposition} Theorem \ref{4fold} is true for $\ol{p}=\ol{0}, \ol{q} = \ol{t}$; and for $\ol{p}=(0,0,1), \ol{q}=(1,2,2)$. \end{proposition}

\begin{proof} If $\alpha \sim_{w, \ol{p}} \alpha'$ and $z$ is a cycle appearing in the equation relating $\alpha$ to $\alpha'$, the dimension of $\pi^{-1}(z^i)$ is at most $2-i + (i-1) =1$.  Since the preimage of the exceptional part is $1$-dimensional, it cannot support a $2$-cycle and the error terms vanish.  Therefore $\widetilde{\alpha} \sim \widetilde{\alpha'}$ and clearly then $\pi_* (\widetilde{\alpha} \cdot \widetilde{\beta}) \sim \pi_* (\widetilde{\alpha'} \cdot \widetilde{\beta})$.

It remains to check the compatibility with $\ol{q}$-rational equivalence in $\beta$, i.e., the vanishing $\pi_* (\widetilde{\alpha} \cdot e_\beta) =0$, where $e_\beta$ is a component of the error term $e(\beta, \beta')$.  Now $\alpha^2$ is a finite set contained in the smooth part $X^2 \setminus X^3$ of $X^2$, and $\alpha^3$ is empty.  If the generic point of $e_\beta$ lies over $X^3$, then $\widetilde{\alpha} \cap e_\beta = \emptyset$ and we are done.

Therefore we suppose the generic point of $e_\beta$ lies over $X^2 \setminus X^3$, and let $E \subseteq \widetilde{E} \into \widetilde{X}$ be a component of the exceptional locus which contains $e_\beta$.  If the image of $E$ has dimension less than or equal to $1$, then $E$ has generic dimension at least $2$ over its image, contradicting the assumption that the generic point of $e_\beta$ lies over $X^2 \setminus X^3$.  Therefore $e_\beta$ first appears on some blowup $\Bl_{S_n}X_n$ where $S_n$ is a smooth surface on some intermediate variety $X_n$.  Now let $E$ denote an exceptional divisor which contains $e_\beta$ (as above), and maps birationally onto a divisor which is exceptional for the morphism $\Bl_{S_n}X_n \to X_n$ (i.e., $E$ also first appears on $\Bl_{S_n}X_n$).

 Let $a : S_n \to X^2$ denote the canonical morphism.  Note that ${a}^{-1}(\alpha^2)$ is finite by our construction of the stratification.
Let $e_n$ denote the (one-dimensional) image of $e_\beta$ in $S_n$.  There exists a $1$-cycle $C$ on $S_n$ which is rationally equivalent to $e_n$, and with the property that $C \cap {a}^{-1}(\alpha^2) = \emptyset$.  In other words, there is a rational function $g_n: S_n \dashrightarrow \bb{P}^1$ such that $[g_n(0)] = e_n + Z$, and such that both $Z$ and $[g_n(\infty)]$ are disjoint from ${a}^{-1}(\alpha^2)$.

Now consider the composition $g: E \to S_n \dashrightarrow \bb{P}^1$.  By the fiber integrality hypothesis (i.e., since (AC) in (4) places the image of the reducible fibers of $E \to S_n$ into $X^3$), the support of $e_\beta$ coincides (set-theoretically) with the preimage of $e_n$ by the morphism $E \to S_n$.  
Therefore $g$ provides a rational equivalence between $m \cdot e_\beta + Z'$ and $Z''$, where $Z'$ is supported over $Z$, and $Z''$ is supported over $[g_n(\infty)]$.  Hence both $Z'$ and $Z''$ are disjoint from $\widetilde{\alpha}$.  Therefore $\widetilde{\alpha} \cdot (m \cdot e_\beta) \sim \widetilde{\alpha} \cdot (Z'' - Z') = 0$ and so $\widetilde{\alpha} \cdot e_\beta$ is zero in $A_0(\widetilde{X})_\bb{Q}$, as desired.

An identical argument handles the pair $\ol{p}=(0,0,1), \ol{q}=(1,2,2)$ since it imposes the same conditions on $\alpha$ and $\beta$. \end{proof}
 
\begin{proposition}  \label{intro smooth ex} Theorem \ref{4fold} is true for $\ol{p}=(0,1,1), \ol{q} = (1,1,2)$. \end{proposition}

\begin{proof}  In this case $\alpha^2$ is finite and $\alpha^4$ is empty, hence if $e(\alpha, \alpha')$ is nonempty it consists of several surfaces lying over some zero-cycle $Z \into X^3 \setminus X^4$.  Thus a component $e_\alpha$ of $e(\alpha, \alpha')$ is an irreducible surface which first appears on $\Bl_{C_n}X_n$, the blowup of some intermediate variety $X_n$ along a smooth one-dimensional center $C_n$.  Let $E$ denote an exceptional divisor that contains $e_\alpha$ and first appears when $e_\alpha$ does (so that $E \into \wt{E}_{C_n}$).
The image of $e_\alpha$ is a point $c \in C_n$; since $C_n$ is finite over its image in $X^2$, the proper transform of $\beta$ via $X_n \to X$ meets $C_n$ in a finite set.

There exists a rational function $g_n: C_n \to \bb{P}^1$ such that $[g_n(0)] = c + Z$, and such that both $Z$ and $[g_n(\infty)]$ are disjoint from the proper transform of $\beta$.  Consider the rational function $g: E \to C_n \to \bb{P}^1$.  Since the image of $c \in C_n$ lies in $X^3 \setminus X^4$, the fiber of $E \to C_n$ over $c$ is irreducible (by (AC) in (5)), and therefore $e_\alpha$ coincides set-theoretically with $E \cap \pi^{-1}(c)$.
Therefore $g$ provides a rational equivalence between $m \cdot e_\alpha + Z'$ and $Z''$, where $Z'$ is supported over $Z \into X^3 \setminus X^4$, and $Z''$ is supported over $[g_n(\infty)]$.  Consequently $e_\alpha \cdot \widetilde{\beta} = 0$ in $A_0(\widetilde{X})_\bb{Q}$.   Note $e_\beta$ must be supported over a finite set in $C_n$, so the same construction moves $e_\alpha$ away from $e_\beta$.

Now we show the vanishing of $\widetilde{\alpha} \cdot e(\beta, \beta')$.  First we consider components $e_\beta$ of $e(\beta, \beta')$ whose generic points lie over $X^2 \setminus X^3$.  Then $e_\beta$ first appears on the blowup of a two-dimensional center $S_n$ which is dominant over a component of $X^2$.  The image of $e_\beta$ in $S_n$ is a subvariety of dimension $1$, and $e_n := p_n (e_\beta)$ is not contracted by $a: S_n \to X^2$.  (If $e_\beta$ were supported over a subvariety contracted by $a$, then it would have generic dimension $2$ over its image.)  We choose an exceptional divisor $E \supset e_\beta$ as usual.

Since $\alpha^4$ is empty, ${a}^{-1}(\alpha^2)$ is finite.  Therefore $e_n$ is rationally equivalent to a $1$-cycle on $S_n$ which is disjoint from ${a}^{-1}(\alpha^2)$, i.e., there exists a rational function $g_n: S_n \dashrightarrow \bb{P}^1$ such that $[g_n(0)] = e_n + Z'$, and such that both $Z'$ and $[g_n(\infty)]$ are disjoint from ${a}^{-1}(\alpha^2)$.  Now we consider the composition $g := g_n \circ p_n : E \to S_n \dashrightarrow \bb{P}^1$.  By the integrality condition on the fibers (condition (AC) in (4)), $[g(0)]$ is a multiple of $e_\beta$ plus a 2-cycle supported over $Z'$, and $[g(\infty)]$ is a 2-cycle supported over $[g_n(\infty)]$.  In particular, some multiple of $e_\beta$ is rationally equivalent to a cycle disjoint from $\widetilde{\alpha}$.

Next we consider error terms $e_\beta$ whose generic points are supported over $X^3 \setminus X^4$.  There are two possibilities for first appearance.
Such a term may first appear on the blowup along a (smooth) one-dimensional center $C_n$ (with one-dimensional image in $X^2$); 
this is handled by finding a rational function on $C_n$ which, upon precomposing with a canonical morphism from an exceptional divisor to $C_n$, gives a rational equivalence between some multiple of $e_\beta$ and a $2$-cycle which is disjoint from the preimage of $\alpha^3$ (as $\alpha^3$ necessarily meets $C_n$ in a finite set) and hence from $\widetilde{\alpha}$.

The other possibility is that the error term $e_\beta$ (with generic points supported over $X^3 \setminus X^4$) first appears on the blowup along a (smooth) two-dimensional center $S_n$, in which case the one-dimensional image of $e_\beta$ inside $S_n$ is contracted to a point by the morphism $S_n \to X^2$: either the center $S_n$ has one-dimensional image in $X^2$, or $e_\beta$ is supported over some exceptional part of the generically finite morphism $S_n \to X^2$.  But $e_\beta$ is exceptional for $S_n \to X^2$ implies the generic point of $e_\beta$ is supported over $X^4$, so we may assume the two-dimensional center $S_n$ has one-dimensional image $W \into X^2$.  By our definition of $X^4$, (the finite sets) $\alpha^3 \cap W$ and $\beta^3 \cap W$ are supported in the smooth locus of $W$, and in the locus over which $S_n \to W$ has integral fibers (by (BC-2) and (BC-3) as in (2)).  Let $q$ denote the canonical morphism $E \to S_n \to W$, and let $w \in W$ be the image of $e_\beta$ via $q$.  Note that $e_\beta$ coincides set-theoretically with $q^{-1}(w)$.  There exists a rational function $g: W \dashrightarrow \bb{P}^1$ such that $[g(0)] = w + Z$, and such that both $Z$ and $[g(\infty)]$ are disjoint from $\alpha^3$.  Then the rational function $E \to S_n \to W \xrightarrow{g} \bb{P}^1$ provides a rational equivalence between some multiple of $e_\beta$ and a cycle which is disjoint from the preimage of $\alpha^3$, so that $\widetilde{\alpha} \cdot e_\beta \sim 0$ as desired.

Finally, components $e_\beta$ of $e(\beta, \beta')$ supported over $X^4$ are automatically disjoint from $\widetilde{\alpha}$.  In particular this applies if $e_\beta$ is supported over some exceptional part of the generically finite morphism $S_n \to X^2$. \end{proof}

\begin{proposition} Theorem \ref{4fold} is true for $\ol{p}=(0,1,2), \ol{q} = (1,1,1)$. \end{proposition}

\begin{proof} In this case $\alpha^2$ is finite (but may have support in $X^4$), $\beta^2$ is $1$-dimensional, $\beta^3$ is finite, and $\beta^4$ is empty.  The error terms appearing in $e(\alpha, \alpha')$ cannot be supported over $X^2 \setminus X^3$ since $p_2=0$.  The error terms supported over $X^3 \setminus X^4$ are handled (moved away from the finite set $\beta^3$) as in the previous case.  The error terms supported over $X^4$ are disjoint from $\widetilde{\beta}$ since $\beta^4 = \emptyset$.  The same reasoning applies with $\wt \beta$ replaced by $e_\beta$.

Now we show the vanishing $\pi_* (\wt \alpha \cdot e_\beta) =0$.
The error terms in $e(\beta, \beta')$ lying over $X^3$ are handled as in the previous case.  It remains to show $\widetilde{\alpha} \cdot e_\beta = 0$ when $e_\beta$ is a component of $e(\beta, \beta')$ whose generic point is supported over $X^2 \setminus X^3$.  We let $e_0$ denote the image of $e_\beta$ in $X^2$, and let $X^2_e$ denote the component of $X^2$ that contains $e_0$.

Let $S$ denote the first two-dimensional center that dominates $X^2_e$.  Then $S \to X^2_e$ is a resolution of singularities, and it (Stein) factors as $S \to S' \to X^2_e$, where $b: S' \to X^2_e$ omits those blowups along centers landing in $X^4$.  (While $S$ is a closed subvariety in some intermediate $X_j$, the variety $S'$ may not admit a closed immersion into any $X_j$.)  Thus $b$ is a finite morphism, the singular set of $S'$ is contained in $b^{-1}(X^4)$, and the proper transform $e_0^{'} \into S'$ of $e_0$ via $S' \to X^2_e$ is supported in the smooth locus of $S'$.

Suppose $e_\beta$ first appears on the blowup of the two-dimensional center $S_n$, and consider the Stein factorization $S_n \to S_n' \to S'$ of the morphism $S_n \to S'$. (Note $S_n \to S$ is dominant since the generic point of $e_\beta$ is supported over $X^2 \setminus X^3$, and of course $S_n=S$ is possible.)   Again choose an exceptional component $E \supset e_\beta$ dominating $S_n$.  Let $e_n$ denote the image of $e_\beta$ in $S_n$, and $e_n'$ its image in $S_n'$.  Since $S_n' \to S' \to X^2_e$ is finite, the preimage of $X^4$ in $S_n'$ is finite.

The morphism $S_n \to S_n'$ contracts exactly those curves lying over curves contracted by $S \to S'$.  Since the curves contracted by $S \to S'$ (more precisely the zero-dimensional image in $X^2$ of such curves) are disjoint from $e_0$, the morphism $S_n \to S_n'$ is an isomorphism in a neighborhood of $e_n$, hence $S_n'$ is smooth in a neighborhood of $e_n'$.  Now we proceed as usual.  We find a rational function on $S_n'$ moving $e_n'$ away from the finite preimage of $X^4$.  Then the composite rational function $E \to S_n \to S_n' \dashrightarrow \bb{P}^1$ moves a multiple of $e_\beta$ away from $\widetilde{\alpha}$, as required.
\end{proof}

\section{Necessity of condition (AC)} \label{need integral fibers}

Since the arguments of the previous section may seem intricate, one may ask if there is a simpler description of a stratification for which intersecting on the resolution induces a well-defined pairing.  In this section we show the stratification according to the fiber dimension of the resolution is insufficiently fine to obtain a well-defined intersection product, even when all the strata are smooth.  The example applies to both equivalence relations $\sim_{\ol p}$ and $\sim_{w, \ol p}$.
We work over a field $k$ of characteristic $\neq 2$.

\medskip
\textbf{An affine example.}  Let $X \into \bb{A}^4$ be the hypersurface defined by the vanishing of $x^2 - y^2 + tz^2$.  Then $X$ is singular along the line $L$ defined by $x=y=z=0$, and we claim a resolution $\pi : \widetilde{X} \to X$ is obtained by blowing up $X$ along $L$.  On the patch where $x$ generates (say $y = y'x$ and $z = z'x$), the blowup is defined by the vanishing of $1- {(y')}^2 + t{(z')}^2$, which is smooth.  On the patch where $y$ generates, the blowup is defined by the vanishing of ${(x')}^2 -1 + t{(z')}^2$, which is also smooth.  On the patch where $z$ generates, the blowup is defined by the vanishing of ${(x')}^2 - {(y')}^2 +t$, and this too is smooth.

Note that all of the fibers of $\pi^{-1}(L) \to L$ are one-dimensional.  Since $L$ is regular, this implies $\pi^{-1}(L) \to L$ is flat and so $L \into X$ is normally flat.  Taking into account only the fiber dimensions in the resolution and the singularities of the centers, we are led to the stratification $X^3 = \emptyset \into X^2 = X^1 = L \into X$.  The relevant feature (which this stratification ignores) is that the fiber of $\pi$ over $(0,0,0,0)$ has two components ($x'+y' = 0$ and $x'-y' =0$), whereas the fiber over $(t_0, 0, 0, 0)$ is irreducible for $t_0 \neq 0$.

Let $D \into X$ be the divisor defined by the ideal $(x+y,t)$, let $\alpha_0$ be the $1$-cycle defined by $(x-y,t,z)$, and let $\alpha_1$ be the $1$-cycle defined by $(x-y,t-1,z)$.  Each of the cycles $D, \alpha_0$, and $\alpha_1$ meets $L$ in a finite set, so $D$ has perversity $\ol 0$ and the $\alpha$'s have perversity $\ol 1 := (1,1,1)$.  Furthermore $\alpha_0 \sim \alpha_1$ since the $\alpha$'s arise as preimages of distinct values of the regular function $t$ on the surface $S$ cut out by $(x-y,z)$.  The equivalence respects the perversity condition $\ol 1$ since $t : S \to \bb{A}^1$ maps $L$ isomorphically onto $\bb{A}^1$, so that each fiber of $t$ meets $L$ exactly once.

Nevertheless, $\widetilde{D} \cap \widetilde{\alpha_0}$ consists of a single (reduced) point: on the patch where $z$ generates, the intersection occurs at $x' =  y' = t = z =0$.  On the other hand, $\widetilde{D} \cap \widetilde{\alpha_1} = \emptyset$ (in fact $D \cap \alpha_1 = \emptyset$). 
We conclude that the stratification is too coarse for the resolution to determine a well-defined intersection product.
Of course the example disappears if we use instead the stratification $X^3 = (0,0,0,0) \into X^2=X^1 = L \into X$, as implied by Theorem \ref{divisor pairing}.

\medskip
\textbf{Behavior at infinity.}  For the reader who does not take degrees of zero-cycles on an affine variety too seriously, we now show the behavior persists on the projective closure.

Let $\ul {X} \into \PP^4_{S,T,X,Y,Z}$ be the hypersurface cut out by $SX^2 - SY^2 + T Z^2$.  The singular locus of $\ul  X$ consists of three components, each abstractly isomorphic to $\PP^1$:
\begin{itemize}
\item $\Sigma_1 = Z(X,Y,Z)$, which is the closure of $L$ in the affine example;
\item $\Sigma_2 = Z(X - Y, S, Z)$; and
\item $\Sigma_3 = Z(X + Y, S, Z)$.
\end{itemize}

The intersection point $p = [0 : 1 : 0 : 0 : 0] = \Sigma_1 \cap \Sigma_2 \cap \Sigma_3$ (which did not appear in the affine example) may be placed in ${\ul X}^3$ (because the singular locus of $\ul  X$ is itself singular there, or because the resolution has different behavior over $p$).  Our constructions will preserve any perversity condition which forbids incidence with $p$, hence will show
$${\ul X}^3 = p \into {\ul X}^2 = {\ul X}^1 = \Sigma_1 \cup \Sigma_2 \cup \Sigma_3 \into \ul X$$
is not fine enough for the intersection product to be well-defined.  As in the affine example, ${\ul X}^3$ must include $[1:0:0:0:0]$ as well.

A resolution is obtained by blowing up $\ul  X$ along $\Sigma_1$, then blowing up the proper transforms $\wt{\Sigma_2}, \wt{\Sigma_3} \into B_{\Sigma_1} (\ul  X)$; since $\wt{\Sigma_2}$ and $\wt{\Sigma_3}$ are disjoint on $B_{\Sigma_1} (\ul  X)$, the order is irrelevant.  Initially blowing up $p$ does not improve the situation, in the sense that one finds a copy of the original singularity on $B_p (\ul X)$.  This is perhaps not too surprising since one finds Whitney umbrellas along $x=0$ and $y=0$, and these too are not resolved by blowing up the ``worst" point in the singular locus.

\begin{proposition} \label{ex strong res} The morphism
$$B_{\wt{\Sigma_2} \cup \wt{\Sigma_3}} (B_{\Sigma_1} (\ul  X ))\to B_{\Sigma_1} (\ul  X) \to \ul  X$$
is a strong resolution of the singularities of $\ul  X$. \end{proposition}
\begin{proof}
First we verify that the blowup along $\Sigma_1 \cup \Sigma_2 \cup \Sigma_3 - p$ resolves $\ul  X - p$; for this calculation the order of the blowup is irrelevant.  Then we analyze the behavior near $p$.

For the blowup along $\Sigma_1 - p$, we work on $S \neq 0$ and recover the affine example: the blowup of $x^2 - y^2 + tz^2$ along $x=y=z=0$.  Hence $B_{\Sigma_1 - p} (\ul  X - p)$ is smooth above $\Sigma_1 - p$.  The analysis of the other blowups is similar (work on $X \neq 0$) and we omit the details.

Having analyzed the morphism away from $p$, we analyze the blowup along $\Sigma_1$ where $T \neq 0$, so that we are blowing up $sx^2 - sy^2 + z^2 = 0$ along $x=y=z=0$.  Where $x$ generates, the blowup is defined by $s - s {(y')}^2 + {(z')}^2 =0$, and is singular along $1-{(y')}^2 = s = z' = 0$; these two singular points are exactly where the proper transforms $\wt{\Sigma_2}, \wt{\Sigma_3}$ meet the fiber over $p$.  Where $y$ generates, the blowup is defined by $s{(x')}^2 -s + {(z')}^2 =0$, and is singular along ${(x')}^2-1=s=z'=0$ (so that we see the same two points).  Where $z$ generates, the blowup is smooth.

The blowup of $B_{\Sigma_1}(\ul  X)$ along $\wt{\Sigma_2}$ is covered by two charts:
\begin{itemize}
\item the blowup of $s - s{(y')}^2 + {(z')}^2 =0$ along $1-y' = s = z' = 0$, and
\item the blowup of $s{(x')}^2 -s + {(z')}^2 =0$ along $x'-1 = s= z' =0$;
\end{itemize}
both of these are smooth.  The blowup along $\wt{\Sigma_3}$ is similar, except one uses the centers $1+y' = s = z' = 0$, and $x'+1 = s= z' =0$ in these charts.
All three centers are normally flat in their ambient spaces since the exceptional divisors over them are irreducible, and the centers are regular and one-dimensional.
\end{proof}

\begin{lemma} 
The stratification
$${\ul X}^3 = p \into {\ul X}^2 = {\ul X}^1 = \Sigma_1 \cup \Sigma_2 \cup \Sigma_3 \into \ul X$$
satisfies (BC-1), (BC-2), and (BC-3) with respect to the resolution of Proposition \ref{ex strong res}.  With respect to this stratification,
\begin{itemize}
\item the divisor $\ul  D = Z(X+Y, T) \into \ul X$ satisfies the perversity condition $\ol 0$, and
\item the 1-cycles $\ul  \alpha_0 = Z(X-Y,T,Z)$ and $\ul  \alpha_1 = Z(X-Y, T-S,Z)$ satisfy the perversity condition $\ol 1$ and are equivalent as 1-cycles of perversity $\ol 1$.
\end{itemize} 
Nevertheless, $\deg (\pi_* ( \wt{\ul  D} \cdot \wt{\ul  \alpha_0}) ) =1$ and $\wt{\ul  D} \cap \wt{\ul  \alpha_1} = \emptyset$.
\end{lemma}
\begin{proof}
The first claim follows from the following incidence properties: $\ul  D \cap \Sigma_1 = [ 1 : 0 : 0 : 0 : 0 ] \ , \ \ul  D \cap \Sigma_2 = \emptyset \ , \ \ul  D \cap \Sigma_3 = [ 0 :0:1:1:0]$.

Next we analyze the rational equivalence relating $\ul  \alpha_0 = Z(X-Y,T,Z)$ and $\ul  \alpha_1 = Z(X-Y, T-S,Z)$.  Both of these 1-cycles are lines in $\PP^2 \cong Z(X - Y, Z) =: P \into \ul  X$.  Note $\Sigma_1 \cup \Sigma_2 \into P$, so we must ensure the equivalence $\ul  \alpha_0 \sim \ul  \alpha_1$ can be chosen to respect the perversity condition.  Since $\Sigma_3 \cap P = p$, the component $\Sigma_3$ poses no difficulty beyond that presented by $\Sigma_1 \cup \Sigma_2$.

We claim there exists an $\bb{A}^1$-relative 1-cycle $A \into \ul  X \times \bb{A}^1$ with the following properties:
\begin{enumerate}
\item $A_0 = \ul \alpha_0, A_1 = \ul  \alpha_1$;
\item all of the fibers $A_t$ are disjoint from $p$; and
\item all of the fibers $A_t$ meet $\Sigma_1 \cup \Sigma_2$ in finitely many points (in fact, in exactly two points).
\end{enumerate}
This means exactly that $A$ determines an $\bb{A}^1$-family of 1-cycles of perversity $\ol 1$, so that $\ul \alpha_0 \sim_{\ol 1} \ul \alpha_1$.  (Of course $\ul \alpha_0$ and $\ul \alpha_1$ are then weakly equivalent as 1-cycles of perversity $\ol 1$.)

We consider $[\ul  \alpha_0], [ \ul  \alpha_1], [\Sigma_1],$ and $[\Sigma_2]$ as points of $\check P$, the $\PP^2$ dual to $P$.  The lines $ \ul \alpha_0, \ul  \alpha_1, \Sigma_2$ meet at $[0:0:1:1:0]$, and this point does not belong to $\Sigma_1$.  Therefore $[\ul  \alpha_0], [ \ul  \alpha_1],$ and $[\Sigma_2]$ lie on a line $\ell \into \check P$, and $[\Sigma_1] \notin \ell$.  Then the family $A$ is the family of lines corresponding to the canonical morphism $\ell - [\Sigma_2] \to \check P$.  Since this family avoids $[\Sigma_1]$ and $[\Sigma_2]$, all of the fibers $A_t$ meet $\Sigma_1 \cup \Sigma_2$ exactly twice.  Furthermore, exactly one point $[L] \in \ell$ corresponds to a line containing $p$; this is $[\Sigma_2]$, hence all of the fibers $A_t$ are disjoint from $p$.

The incidences are contained in the locus where $S \neq 0$, so are captured by the affine situation.
\end{proof}

\begin{remark}
In this example, the family $A$ of 1-cycles cannot be extended to a $\PP^1$-family respecting the perversity condition $\ol 1$.  However, by performing the blowups in a different order, one can find an example of the same essential phenomenon, and so that the $\bb{A}^1$-family relating the 1-cycles extends to a $\PP^1$-family, all of whose fibers satisfy the perversity condition $\ol 1$.  Namely, we first blow up $\Sigma_2$ on $\ul X$, then blow up the proper transform of $\Sigma_3$.  The resulting variety $\ul X'$ is singular exactly along ${\Sigma'_1}$, the proper transform of $\Sigma_1$.  There are no incidences $\ul D' \cap \ul \alpha_i'$ in the exceptional divisors over $\Sigma_2 \cup \Sigma_3$, so the incidences are captured by the affine situation.  The variety $\ul X'$ is resolved by blowing up $\Sigma_1$, and $B_{\Sigma'_1} (\ul X')$ has one-dimensional fibers over $\Sigma'_1$, all of which are irreducible except one (the same which appears in the affine example).

Note $\Sigma_2 \into P$ is already Cartier, so the proper transform of $P$ via the blowup along $\Sigma_2$ is isomorphic to $P$.  Since $\Sigma_3$ meets $P$ in a single point, the transform $P'$ of $P$ on $\ul X'$ is isomorphic to $\PP^2$ blown up at a single point.  

The blown up point is not contained in any of the lines $\Sigma_1', \ul \alpha'_0, \ul \alpha'_1$.  The rational function $P \dashrightarrow \PP^1$ relating $\ul \alpha_0$ to $\ul \alpha_1$ may be considered as a rational function $P' \dashrightarrow \PP^1$ relating $\ul \alpha'_0$ to $\ul \alpha'_1$; since $\Sigma'_1$ does not occur as a fiber of this map, we may think of this function as a $\PP^1$-family of 1-cycles of perversity $\ol 1$ for the stratification ${(\ul X')}^3 = \emptyset \into {(\ul X')}^2 = {(\ul X')}^1 = \Sigma'_1 \into \ul X'$.
\end{remark}

\section{Further questions} \label{sec:further}

\textbf{Independence of resolution.}  Given two resolutions $\pi_1 :X_1 \to X \ , \ \pi_2 : X_2 \to X$ of the variety $X$, we have defined two stratifications $S_1, S_2$ of $X$, and (in certain situations) products $\bu_i : A_{r, \ol p}(X, S_i) \otimes A_{s, \ol q}(X, S_i) \to A_{r+s-d}(X)$ (for $i=1,2$) by pushing forward the intersection formed on the resolution.  (We make explicit the dependence of the group $A_{r, \ol p}(X)$ on the stratification, and for simplicity we drop the superscript $w$ and the coefficients.)  For any stratification $S$ that refines both $S_1$ and $S_2$, we have canonical morphisms $C^i_{r, \ol p}: A_{r,  \ol p}(X, S) \to A_{r, \ol p}(X, S_i)$ (for $i=1,2$), and thus it makes sense to ask whether $\bu_1 (C^1_{r, \ol p} \otimes C^1_{s, \ol q})$ and $\bu_2 (C^2_{r, \ol p} \otimes C^2_{s, \ol q} )$ coincide as morphisms $A_{r,  \ol p}(X, S) \otimes A_{s, \ol q}(X, S) \to A_{r+s-d}(X)$.  We assume we are in one of the situations in which the resolution is known to induce a well-defined product.

The simplest case is when $\pi_2$ is obtained from $\pi_1$ by a sequence of blowups of $X_1$ along smooth centers.  In this case we have a canonical morphism 
$C_{r, \ol p} : A_{r,  \ol p}(X, S_2) \to A_{r, \ol p}(X, S_1)$ (i.e., the stratification induced by $\pi_2$ refines the one induced by $\pi_1$).  

\begin{proposition} With the notation and hypotheses as above, suppose $f : X_2 \to X_1$ is a composition of blowups along smooth centers, and set $\pi_2 = \pi_1 \circ f : X_2 \to X$.  Then we have $\bu_2 = \bu_1 (C_{r, \ol p} \otimes C_{s, \ol q}) : A_{r, \ol p}(X,S_2) \otimes A_{s, \ol q}(X, S_2) \to A_{r+s-d}(X)$. \end{proposition}

\begin{proof}  Let $(-)_i$ denote the proper transform of a cycle on $X$ via $\pi_i : X_i \to X$.  We have $f^* (\alpha_1) = \alpha_2 + e_\alpha$ and $f^* (\beta_1) = \beta_2 + e_\beta$.  Since $\alpha_1 \cdot \beta_1 = f_* (f^* (\alpha_1) \cdot f^* (\beta_1))$, it follows that
$${(\pi_1)}_* (\alpha_1 \cdot \beta_1) = {(\pi_2)}_* (\alpha_2 \cdot \beta_2) + {(\pi_2)}_*(\alpha_2 \cdot e_\beta + e_\alpha \cdot \beta_2 + e_\alpha \cdot e_\beta).$$
The cycles $e_\alpha$ and $e_\beta$ may be thought of as error terms which first appear on one of the blowups occurring in the morphism $f$.  Therefore our arguments apply to show the vanishing of the cycle class of ${(\pi_2)}_*(\alpha_2 \cdot e_\beta + e_\alpha \cdot \beta_2 + e_\alpha \cdot e_\beta)$, and therefore we obtain the equality ${(\pi_1)}_* (\alpha_1 \cdot \beta_1) = {(\pi_2)}_* (\alpha_2 \cdot \beta_2).$
\end{proof}

If $k$ is an algebraically closed field of characteristic zero and $X$ is complete, then two resolutions of $X$ are related by a sequence of smooth blowups and blowdowns by the weak factorization theorem of Abramovich-Karu-Matsuki-W{\l}odarczyk \cite[Thm.~0.1.1]{AKMW}.  Since the intermediate varieties do not necessarily admit morphisms to $X$, it is not clear how to obtain a comparison of the products formed via two resolutions.  If the strong factorization conjecture \cite[Conj~0.2.1]{AKMW} holds, however, then there is variety $Y$ admitting morphisms $f_i : Y \to X_i$ which are compositions of blowups along smooth centers, and such that $\pi_1 \circ f_1 = \pi_2 \circ f_2 : Y \to X$.  In this case we could conclude that the products defined using $X_1$ and $X_2$ agree upon restriction to $A_{r, \ol p}(X, S_Y) \otimes A_{s, \ol q}(X, S_Y)$, i.e., the group formed using the stratification induced by the resolution $Y \to X$.

\medskip
\textbf{Comparison with Goresky-MacPherson product.} It would be interesting to know whether our intersection product (when it is defined) agrees with the intersection pairing defined by Goresky-MacPherson via the cycle class mapping.  The difficulty in making the comparison is that there is no obvious way to take the proper transform of a topological cycle.  A cycle on $X$ gives rise to a canonical cycle on $\wt X$ relative to $\wt E$, and using this one can show the products agree after composing with the canonical map $H_*(X) \to H_*(X, X^1)$.  For pairs of cycles with supports intersecting properly in each stratum, or more generally any pair which is weakly $(\sim_{\ol p}, \sim_{\ol q})$-equivalent to such a pair, our intersection product agrees with that of Goresky-MacPherson since in this case both may be described as taking the closure of an intersection product formed on the smooth locus.

\medskip
\textbf{Refinements for small perversities.} If $\ol p + \ol q < \ol t$, is there a refinement
$$A_{r, \ol p}(X) \otimes A_{s, \ol q}(X) \to A_{r+s - d, \ol p + \ol q}(X) \ \ ?$$
This seems difficult to achieve using resolutions.  For example, suppose $\alpha \in A_{r, \ol 0}(X)$ and $D \in A_{d -1, \ol 0}(X)$.  The fibers of $\wt \alpha \cap \pi^{-1}(X^i) \to \alpha \cap X^i$ are typically $(i-1)$-dimensional (the source is typically $(r-1)$-dimensional, and the target is typically $(r-i)$-dimensional).  To find a representative of $\pi_* (\wt \alpha \cdot \wt D)$ in $A_{r-1, \ol 0}(X)$, one would seek a divisor $\wt D_1 \sim \wt D \into \wt X$ such that the image of $\wt \alpha \cap \pi^{-1}(X^i) \cap \wt D_1 \to \alpha \cap X^i$ is not dense.  This means $\wt D_1$ misses most fibers of the morphism $\wt \alpha \cap \pi^{-1}(X^i) \to \alpha \cap X^i$.  One might try the following technique for moving divisors, at least in the quasi-projective case: find some effective divisor $D$ (with better incidence properties) such that $\wt D + D$ is ample, then use $\wt D_1 = H-D$ for some $H \in |\wt D + D|$.
Unfortunately, the divisor $H$, being ample, will meet \textit{every} fiber of the morphism $\wt \alpha \cap \pi^{-1}(X^i) \to \alpha \cap X^i$ if $i-1 \geq 1$.

\bibliography{intlawhom}{}
\bibliographystyle{plain}

\end{document}